\documentclass[final,2p,times]{elsarticle}

\usepackage{bbm}
\usepackage{amsxtra,amsmath,amssymb,amsfonts,exscale,amsthm,amscd}
\usepackage{xcolor}
\usepackage{float}
\usepackage{epic}
\usepackage{latexsym}
\usepackage{geometry}
\usepackage{graphicx}
\usepackage{subfigure}
\usepackage{natbib}
\usepackage{txfonts}
\usepackage{psfrag}
\usepackage{showkeys}
\usepackage{pifont}
\usepackage{ifthen}
\usepackage{ifpdf}
\usepackage{slashbox}
\usepackage{tabularx}
\usepackage{booktabs}
\usepackage{dcolumn}
\usepackage{multirow}
\usepackage{wrapfig}
\usepackage{multicol}
\usepackage{verbatim}
\usepackage[shortlabels]{enumitem}
\usepackage{layouts}
\usepackage{stmaryrd}
\usepackage{xcolor}
\usepackage{tikz}
\usepackage{diagbox}
\usepackage{mathrsfs}
\usepackage{ulem}

\providecommand{\Div}{\operatorname{div}}          
\providecommand{\curl}{\operatorname{{\bf curl}}}  

\providecommand*{\Dist}[2]{\operatorname{dist}({#1};{#2})}   
\providecommand*{\Dist}[2]{\Dist{#1}{#2}}

\newcommand{\Bf}{{\boldsymbol{f}}}
\newcommand{\Bg}{{\boldsymbol{g}}}

\newcommand{\Bn}{{\boldsymbol{n}}}

\newcommand{\Bp}{{\boldsymbol{p}}}

\newcommand{\Bu}{{\boldsymbol{u}}}
\newcommand{\Bv}{{\boldsymbol{v}}}
\newcommand{\Bw}{{\boldsymbol{w}}}
\newcommand{\Bx}{{\boldsymbol{x}}}
\newcommand{\By}{{\boldsymbol{y}}}
\newcommand{\Bz}{{\boldsymbol{z}}}

\newcommand{\BA}{{\boldsymbol{A}}}

\newcommand{\BC}{{\boldsymbol{C}}}

\newcommand{\BH}{{\boldsymbol{H}}}

\newcommand{\BL}{{\boldsymbol{L}}}

\newcommand{\BU}{{\boldsymbol{U}}}
\newcommand{\BV}{{\boldsymbol{V}}}

\newcommand{\BX}{{\boldsymbol{X}}}

\newcommand{\xibf}{\boldsymbol{\xi}}

\newcommand{\chibf}{\boldsymbol{\chi}}

\newcommand{\Ce}{\mathcal{E}}

\newcommand{\Cl}{\mathcal{L}}

\newcommand{\Cp}{\mathcal{P}}

\newcommand{\Ct}{\mathcal{T}}

\newcommand{\bbR}{\mathbb{R}}

\newcommand*{\N}[1]{\left\|{#1}\right\|}     
\newcommand*{\SN}[1]{\left|{#1}\right|}      

\newcommand*{\Lp}[2][\defaultdomain]{L^{#2}({#1})}

\newcommand*{\Ltwo}[1][\defaultdomain]{\Lp[#1]{2}}

\newcommand*{\Hm}[2][\defaultdomain]{H^{#2}({#1})}
\newcommand*{\Hmv}[2][\defaultdomain]{\BH^{#2}({#1})}

\newcommand*{\Hone}[1][\defaultdomain]{\Hm[#1]{1}}
\newcommand*{\Honev}[1][\defaultdomain]{\Hmv[#1]{1}}

\newcommand*{\jump}[2][]{\left \llbracket{#2}\right\rrbracket_{#1}}

\newcommand{\D}{\mathrm{d}}

\newcommand{\ol}{\overline}
\newcommand{\ul}{\underline}

\newcommand{\be}{\begin{eqnarray}}
	\newcommand{\ee}{\end{eqnarray}}

\newcommand{\ben}{\begin{eqnarray*}}
	\newcommand{\een}{\end{eqnarray*}}

\newtheorem{theorem}[subsection]{\sc Theorem}

\newtheorem{remark}[subsection]{\sc Remark}
\newtheorem{algorithm}[subsection]{\sc Algorithm}
\newtheorem{example}[subsection]{\sc Example}

\journal{Journal of Computational Physics}

\begin{document}

\begin{frontmatter}

\title{A high-order finite element method for nonlinear convection-diffusion equation on time-varying domain}

\author[add1,add2]{Chuwen Ma}
\ead{chuwenii@lsec.cc.ac.cn}
\author[add1,add2]{Weiying Zheng\corref{cor}\fnref{zwy}}
\ead{zwy@lsec.cc.ac.cn}
\address[add1]{LSEC, NCMIS, Institute of Computational Mathematics and Scientific/Engineering Computing, Academy of Mathematics and Systems Science, Chinese Academy of Sciences, Beijing, 100190, China.}
\address[add2]{School of Mathematical Science, University
of Chinese Academy of Sciences, Beijing, 100049, China.}
\cortext[cor]{Corresponding author: Tel: +86-10-82541735.}

\fntext[zwy]{This author was supported in part by the National Science Fund for Distinguished Young Scholars 11725106 and by China NSF grant 11831016.}

\begin{abstract}
A high-order finite element method is proposed to solve the nonlinear convection-diffusion equation on a time-varying domain whose boundary is implicitly driven by the solution of the equation. The method is semi-implicit in the sense that the boundary is traced explicitly with a high-order surface-tracking algorithm, while the convection-diffusion equation is solved implicitly with high-order backward differentiation formulas and fictitious-domain finite element methods.
By numerical experiments on severely deforming domains, we show that optimal convergence orders are obtained in energy norm for both third-order and fourth-order methods.
\end{abstract}

\begin{keyword}
Nonlinear convection-diffusion equation, free-surface problem, SBDF scheme, fictitious-domain finite element method, surface-tracking algorithm.
\vspace{2mm}

\end{keyword}
\end{frontmatter}

\section{Introduction}
Time-varying interface problems are frequently encountered in numerical simulations for multi-phase flows and fluid-structure interactions.
Transient deformation of material region poses a great challenge to the design of
high-order numerical methods for solving such kind of problems. The problem becomes even more difficult when the evolution of the solution and the deformation of the domain have strong and nonlinear interactions.

In this paper, we study the nonlinear convection-diffusion equation on a time-varying domain
\begin{subequations}\label{cd-model}
\begin{align}	
\partial_t\Bu+ \Bu \cdot \nabla \Bu- \nu\Delta\Bu  = \Bf&
 	\quad \text{in}\;\; \Omega_t, \label{cd-eqn}   \\
\partial_\Bn \Bu =0& \quad \text{on}\;\;\Gamma_t, \label{cd-bc}\\
\Bu(0) =\Bu_0& \quad \text{in}\;\;\Omega_0,  \label{cd-ic}
\end{align}	
\end{subequations}
where $\Omega_t\subset \bbR^2$ is a bounded domain for each $t\ge 0$,
$\Gamma_t = \partial \Omega_t$ the boundary of $\Omega_t$, $\Bf(\Bx,t)$ the source term, $\nu$ the diffusion coefficient which is a positive constant, $\Bu_0$ the initial value, and $\Omega_0$ the initial shape of the domain. Moreover, $\partial_t\Bu$ denotes the partial derivative $\frac{\partial\Bu}{\partial t}$ and $\partial_\Bn\Bu$ denotes the directional derivative $\frac{\partial\Bu}{\partial\Bn}$ on $\Gamma_t$ with $\Bn$ being the unit outer normal on $\Gamma_t$. Equation \eqref{cd-eqn} can be viewed as a simplification of the momentum equation of the incompressible Navier-Stokes equations by removing the pressure term.

The flow velocity $\Bu$ drives the deformation of $\Omega_t$, that is,
$\Omega_t = \big\{\BX(t;0,\Bx):
\forall\,\Bx\in\Omega_0\big\}$,
where $\BX(t;0,\Bx)$ is the solution to the
ordinary differential equation (ODE)
\begin{equation}\label{eq:X}
\frac{\D}{\D t} \BX(t;0,\Bx)= \Bu\big(\BX(t;0,\Bx),t\big),\qquad
\BX(0;0,\Bx) = \Bx.
\end{equation}
Throughout the paper, we assume that the exact solution $\Bu(\Bx,t)$ is $\BC^{2}$-smooth in both $\Bx$ and $t$. Classic theories for ODEs show that \eqref{eq:X} has a unique solution.
The nonlinearity of the problem appears not only in the convection term, but also in the interaction between $\Bu$ and $\Omega_t$.
To design high-order numerical methods for \eqref{cd-model}, one
must also seek a high-order surface-tracking algorithm for building the varying domain with numerical solutions.
An entirely high-order solver requires that the surface-tracking algorithm has at least the same order of accuracy as the numerical method for solving the partial differential equation (PDE) \cite{ma21}.

If the driving velocity of surface is known, there are extensive works in the literature on tracing and representing the surface approximately. We refer to \cite{ale99,gig06,her08,mor17,osh88,osh00} and the references therein for level set methods, to \cite{har00,hir81,kum21} for volume-of-fluid methods,
to \cite{ahn07,ahn09,dya08} for moment-of-fluid methods, and
to \cite{try01,zha01} for front-tracking method. In a series of works \cite{zha14,zha16,zha17,zha18}, Zhang developed the cubic MARS (Mapping and Adjusting Regular Semi-analytic sets) algorithm for tracking the loci of free surface. The algorithm traces control points on the surface and
forms an approximate surface with cubic spline interpolations.
It can achieve high order accuracy by controling distances between neighboring control points.
In the previous work \cite{ma21}, we developed a high-order finite element method for solving linear advection-diffusion equation on Eulerian meshes. In that work, the driving velocity of the domain is explicitly given so that the surface-tracking procedure can be implemented easily with cubic MARS algorithm. A thorough error analysis for the finite element solution is conducted by considering all errors from interface-tracking, spatial discretization, and temporal integration. However, for the nonlinear problem \eqref{cd-model}, what's more difficult is that {\it the driving velocity of the boundary is the unknown solution to the PDE}. To maintain the overall high-order accuracy of numerical method, we propose a high-order implicit-explicit (IMEX) scheme which advances the boundary explicitly, while solves \eqref{cd-model} implicitly.

During the past three decades, numerical schemes for solving interface problems on unfitted grids are very popular in the literature. To mention some of them, we refer
to \cite{bre12,mit05,pes77} for the immersed boundary method, to \cite{lev94,li06} for the immersed interface method, to \cite{guo21,li98,li03,lin09} for the immersed finite element method, and to \cite{bec09,han02,hua17,leh15,liu20,leh13,wu19} for Nitsche extended finite element methods. The main idea is to double the degrees of freedom  on interface elements and add penalty terms to enforce interface conditions weakly.
Similar ideas can also be found in fictitious domain methods which allow the boundary to cross mesh elements  \cite{bur10,bur12,jar09,mas14}. To the best of our knowledge, there are very few papers in the literature on high-order methods for free-surface problems where the motion of domain is driven by the solution to the equation, particularly, for those problems on severely deforming domains. The purpose of this paper is to develop third- and fourth-order methods for
solving the nonlinear convection-diffusion equation \eqref{cd-model} on time-varying domains.

For time integration, we propose an IMEX scheme by applying the $k^{\rm th}$-order Semi-implicit Backward Difference Formula (SBDF-$k$) \cite{asc95} to a Lagrangian form of the convection-diffusion equation. The spatial arguments of the velocity are implicitly defined by means of flow maps.
For spatial discretization, we adopt high-order fictitious-domain finite element methods using cut elements. The methods are based on a fixed Eulerian mesh that covers the full movement range of the deforming domain.
The novelties of this work are listed as follows.
\begin{itemize}[leftmargin=5mm]
\item [1.] A high-order IMEX scheme is proposed for solving \eqref{cd-model}, where the time integration is taken implicitly along characteristic curves, while the computational domains are formed explicitly with high-order surface-tracking algorithm.

\item [2.] By numerical examples on severely deforming domains, we show that the SBDF-$3$ and  SBDF-$4$ schemes have optimal convergence orders.
\end{itemize}
\vspace{2mm}

The rest of the paper is organized as follows.
In section~2, we introduce the semi-discrete SBDF-$k$ schemes, $1\le k\le 4$, for solving problem \eqref{cd-model}. In section~3, we
introduce the fully discrete finite element method for solving \eqref{cd-model} on a fixed Eulerian mesh. A high-order surface-tracking algorithm is presented to build the computational domains explicitly with numerical solutions. In section~4, we present an efficient algorithm for computing backward flow maps.
Three numerical examples are presented to show that the proposed numerical methods
have overall optimal convergence orders for $k=3,4$.

\section{The semi-discrete schemes}

For any fixed $\Bx\in\Omega_0$, the chain rule and \eqref{eq:X} indicate that the material derivative of $\Bu$ satisfies
\begin{equation}\label{eq:DuDt}
\frac{\D}{\D t}\Bu\big(\BX(t;0,\Bx),t\big) = \partial_t \Bu\big(\BX(t;0,\Bx),t\big)
+ \Bu\big(\BX(t;0,\Bx),t\big) \cdot\nabla \Bu\big(\BX(t;0,\Bx),t\big),
\end{equation}
where $\nabla\Bu(\Bz,t)$ denotes the gradient of $\Bu$ with respect to $\Bz$ and $\partial_t\Bu(\Bz,t)$ denotes the partial derivative of $\Bu$ with respect to $t$. Similarly, $\Delta\Bu(\Bz,t)$ also denotes the Laplacian of $\Bu$ with respect to $\Bz$. For convenience, we omit the arguments of $\Bu$ without causing confusions and write problem \eqref{cd-model} into a compact form
\begin{subequations}\label{cd1}
\begin{align}	
\frac{\D\Bu}{\D t}- \nu\Delta\Bu  = \Bf&
 	\quad \text{in}\;\; \Omega_t, \label{cd1-eqn}   \\
\partial_\Bn \Bu =0& \quad \text{on}\;\;\Gamma_t, \label{cd1-bc}\\
\Bu|_{t=0} =\Bu_0& \quad \text{in}\;\;\Omega_0.  \label{cd1-ic}
\end{align}	
\end{subequations}
Through \eqref{eq:DuDt}, problems \eqref{eq:X} and \eqref{cd1} form a coupled system of initial-boundary value problems.

Now we describe the SBDF-$k$ for solving \eqref{eq:X} and \eqref{cd1}.
Let $0=t_1\le t_2\le \cdots \le t_N=T$ be the uniform partition of the interval $[0,T]$ with step size $\tau=T/N$. For convenience, we denote the exact flow map from $\Omega_{t_m}$ to $\Omega_{t_n}$ by $\BX^{m,n}:=\BX(t_n;t_m,\cdot)$ for all $0\le m\le n\le N$.
The inverse of $\BX^{m,n}$ is denoted by
\ben
\BX^{n,m} := \big(\BX^{m,n}\big)^{-1}.
\een
For each $k\le n\le N$, we are going to study the approximate solution $\BX^{0,n}_\tau(\Bx)$ of \eqref{eq:X} and the approximate solution $\Bu^n$ of \eqref{eq:DuDt} at $t_n$.

First we assume that the discrete approximations $\BX^{0,m}_\tau(\Bx)$ of $\BX(t_m;0,\Bx)$ have already been obtained for all $0\le m<n$ and all $\Bx\in\Omega^0\equiv \Omega_0$. The approximate domain at $t_m$ is defined as the range of $\BX^{0,m}_\tau$, namely,
\ben
\Omega^m=\big\{\BX^{0,m}_\tau(\Bx): \forall\,\Bx\in\Omega^0\big\}.
\een
Suppose that $\BX^{0,m}_\tau$ provides a good approximation to $\BX^{0,m}$.
The boundedness of $\BX^{0,m}$ implies that $\BX^{0,m}_\tau$ is one-to-one and has a bounded inverse.
For any $0\le j\le m$, the maps $\BX^{j,m}_\tau$: $\Omega^{j}\to \Omega^m$ and
$\BX^{j,m}$: $\Omega_{t_j}\to \Omega_{t_m}$ are defined as
\begin{equation}\label{Xmn}
\BX_\tau^{j,m} := \BX_\tau^{0,m}\circ \big(\BX_\tau^{0,j}\big)^{-1},\qquad
\BX^{j,m} := \BX^{0,m}\circ \big(\BX^{0,j}\big)^{-1}.
\end{equation}
The inverses of $\BX_\tau^{j,m}$ is denoted by
\begin{equation}\label{Xmj}
\BX_\tau^{m,j} := \big(\BX_\tau^{j,m}\big)^{-1}.
\end{equation}

Next we suppose that the approximate solutions $\Bu^m$ of \eqref{cd1} are obtained for $0\le m<n$ and that each $\Bu^m$ is supported on $\ol{\Omega^m}$. In the $n^{\rm th}$ time step, the explicit $k^{\rm th}$-order scheme for solving \eqref{eq:X} has the form
\begin{equation}\label{eq:DnX0}
a_0^k \BX_\tau^{0,n}(\Bx)
= \tau \sum_{i =1}^{k} b_i^k \Bu^{n-i}\circ\BX_\tau^{0,n-i}(\Bx)
-\sum_{i=1}^k a_i^k \BX_\tau^{0,n-i}(\Bx), \qquad \forall\,\Bx\in \Omega^0.
\end{equation}
where the coefficients $a_i^k,b_i^k$ are listed in Table~\ref{tab:SBDF}.
Using \eqref{Xmj}, we can rewrite \eqref{eq:DnX0} equivalently as follows
\begin{equation}\label{eq:DnX}
a_0^k \BX_\tau^{n-1,n}(\Bx)
= \tau \sum_{i =1}^{k} b_i^k \Bu^{n-i}\circ\BX_\tau^{n-1,n-i}(\Bx)
-\sum_{i=1}^k a_i^k \BX_\tau^{n-1,n-i}(\Bx), \qquad \forall\,\Bx\in \Omega^{n-1}.
\end{equation}
Clearly $\BX^{n-1,n}_\tau$ defines the approximate domain of the $n^{\rm th}$ time step
\begin{equation}\label{Dn}
\Omega^n=\big\{\BX^{n-1,n}_\tau(\Bx): \forall\,\Bx\in\Omega^{n-1}\big\},\qquad
\Gamma^n:=\partial\Omega^n.
\end{equation}
Define $\Bf^n:=\Bf(\cdot,t_n)$. The implicit $k^{\rm th}$-order scheme for solving \eqref{cd1} is given by
\begin{subequations}\label{un-pro}
\begin{align}
a_0^k\Bu^n-\tau\nu\Delta\Bu^{n}=\tau\Bf^n - \sum_{i=1}^ka_i^k\Bu^{n-i}
\circ\BX^{n,n-i}_\tau\quad
&\hbox{in}\;\;\Omega^n, \label{un-eqn}\\
\partial_\Bn\Bu^n =0 \quad
&\hbox{on}\;\;\Gamma^n. \label{un-bc}
\end{align}
\end{subequations}

\begin{table}[http!]
\center
\caption{Coefficients for SBDF schemes.}\label{tab:SBDF}
\setlength{\tabcolsep}{2mm}
\begin{tabular}{ |c|l|l|l|l|l|}
\hline
\diagbox[dir=SE]{$k$}{$\big(a_i^k,b_i^k\big)$}{$i$}  &\qquad $0$  &\quad\;\;$1$  &\qquad $2$  &\qquad $3$   &\qquad $4$  \\  \hline
$1$      & $(1,\times)$       &$(-1,\,1)$  &$(0,\,0)$       &$(0,\,0)$     &$(0,\,0)$  \\ \hline
$2$      &$(3/2,\,\times)$    &$(-2,\,2)$  &$(1/2,\,-1)$    &$(0,\,0)$     &$(0,\,0)$    \\ \hline
$3$      &$(11/6,\,\times)$   &$(-3,\,3)$  &$(3/2,\,-3)$    &$(-1/3,\,1)$  &$(0,\,0)$     \\ \hline
$4$      &$(25/12,\,\times)$  &$(-4,\,4)$  &$(3,\,-6)$      &$(-4/3,\,4)$  &$(1/4,\,-1)$  \\ \hline
\end{tabular}
\end{table}

\section{High-order finite element methods}
\label{sec:fem}

In practice, we are not able to build $\Omega^n$ with \eqref{eq:DnX} and \eqref{Dn}.
The purpose of this section is to propose a surface-tracking algorithm to build an approximate domain $\Omega^n_h$
and to propose a fictitious-domain finite element method using cut elements (\cite{bur10,bur12}) for computing the discrete solution $\Bu^n_h\in\Honev[\Omega^n_h]$ in each time step.
Let $D\subset\bbR^2$ be an open domain satisfying $\bar\Omega_t\subset D$ for all $0\le t\le T$. Let $\Ct_h$ be the uniform partion of $\bar D$, which consists of close squares of side-length $h$.
\vspace{1mm}

The algorithms for computing $\Bu^n_h$ and $\Omega^n_h$ will be described successively.
\begin{center}
\fbox{\parbox{0.975\textwidth}
{First we assume that all the quantities below have been obtained for all $0\le m<n$,
\begin{enumerate}[leftmargin=5mm]
\item the computational domains $\Omega^m_h$ and the finite element solutions $\Bu^m_h\in\Honev[\Omega^m_h]$,

\item the discrete forward maps $\BX^{m-1,m}_h$: $\bar\Omega^{m-1}_h\to
\BX^{m-1,m}(\bar\Omega^{m-1}_h)$, and

\item the discrete backward maps $\BX^{m,m-1}_h$: $\bar \Omega^{m}_h
\to \bar \Omega^{m-1}_h$.
\end{enumerate}\vspace{1mm}
Our task is to establish the computational domain $\Omega^n_h$, the finite element solution $\Bu^n_h$, the forward map $\BX^{n-1,n}_h$, and the backward map $\BX^{n,n-1}_h$.}}
\end{center}

\subsection{Finite element spaces}

Let $\Gamma^m_h=\partial\Omega^m_h$ denote the boundary of $\Omega^m_h$ for convenience.
We introduce an open domain $\tilde\Omega^m_h$ which is larger than $\Omega^m_h$
\begin{equation}\label{domain-tm}
\tilde\Omega^m_h = \left\{\Bx + t\Bn(\Bx): 0\le t<h/4,\;\Bx\in\Gamma^m_h\right\},
\end{equation}
where $\Bn(\Bx)$ is the unit normal of $\Gamma^m_h$ at $\Bx$ and points to the exterior of $\Omega^m_h$.
Without loss of generality, we assume that $\tilde\Omega^m_h\subset D$.
Let $\Ct^m_{h,I}$ denote the set of interior elements of $\Omega^m_h$, namely,
\ben
\Ct^m_{h,I} = \left\{K\in\Ct_h:\; K\subset \Omega^m_h\right\}.
\een
The mesh $\Ct_h$ induces a cover of $\tilde\Omega^m_h$ and a cover of $\Gamma^m_h$, which are defined as follows
\ben
\Ct^m_h := \big\{K\in\Ct_h:\;
\mathrm{area}(K\cap\tilde\Omega^m_h) >0\big\}, \qquad
\Ct^m_{h,B} := \Ct^m_h\backslash \Ct^m_{h,I} .
\een
The cover $\Ct^m_h$ generates a fictitious domain which is denoted by
\ben
D^m:= \mathrm{interior}
\big(\cup_{K\in\Ct^m_h}K\big).
\een
Clearly we have $\Omega^m_h\subset\tilde\Omega^m_h\subset D^m$.
Let $\Ce_h$ denote the set of all edges in $\Ct_h$. The set of boundary-zone edges is denoted by (see Fig.~\ref{fig:Tnh})
\ben
\Ce_{h,B}^{m}= \left\{E\in\Ce_h: \; E \not\subset\partial D^m
\;\; \hbox{and}\;\; \exists K\in \Ct^m_{h,B}\;\;
\hbox{s.t.}\;\; E\subset\partial K\right\}.
\een

\begin{figure}[http!]
	\centering
	\begin{tikzpicture}[scale =2]
		\filldraw[red!70!white](0.2*3,0.2*2)--(0.2*3,0.2*3)--(0.2*2,0.2*3)--(0.2*2,0.2*7)--(0.2*3,0.2*7)--(0.2*3,0.2*8)--(0.2*4,0.2*8)--(0.2*7,0.2*8)--(0.2*7,0.2*7)--(0.2*8,0.2*7)--(0.2*8,0.2*3)--(0.2*7,0.2*3)--(0.2*7,0.2*2)--(0.2*3,0.2*2);  		
		\filldraw[yellow!70!white](0.2*3,0.2*6)--(0.2*4,0.2*6)--(0.2*4,0.2*7)--(0.2*5,0.2*7)--(0.2*6,0.2*7)--(0.2*6,0.2*6)--(0.2*7,0.2*6)--(0.2*7,0.2*5)--(0.2*7,0.2*4)--(0.2*6,0.2*4)--(0.2*6,0.2*3)--(0.2*5,0.2*3)--(0.2*4,0.2*3)--(0.2*4,0.2*4)--(0.2*3,0.2*4)--(0.2*3,0.2*5)--(0.2*3,0.2*6);
		\filldraw[red!70!white](0.2*1,0.2*4)--(0.2*2,0.2*4)--(0.2*2,0.2*6)--(0.2*1,0.2*6)--(0.2*1,0.2*4);
		\filldraw[red!70!white](0.2*2,0.2*7)--(0.2*3,0.2*7)--(0.2*3,0.2*8)--(0.2*2,0.2*8)--(0.2*2,0.2*7);
		\filldraw[red!70!white](0.2*2,0.2*2)--(0.2*3,0.2*2)--(0.2*3,0.2*3)--(0.2*2,0.2*3)--(0.2*2,0.2*2);
		\filldraw[red!70!white](0.2*7,0.2*2)--(0.2*8,0.2*2)--(0.2*8,0.2*3)--(0.2*7,0.2*3)--(0.2*7,0.2*2);
		\filldraw[red!70!white](0.2*8,0.2*4)--(0.2*9,0.2*4)--(0.2*9,0.2*6)--(0.2*8,0.2*6)--(0.2*8,0.2*4);
		\filldraw[red!70!white](0.2*7,0.2*7)--(0.2*8,0.2*7)--(0.2*8,0.2*8)--(0.2*7,0.2*8)--(0.2*7,0.2*7);
		\draw[black, thick] (1,1) ellipse [x radius=0.56cm, y radius=0.5cm];
	    \draw[yellow,thick] (1,1) ellipse [x radius=0.62cm, y radius = 0.56cm];
		\filldraw[step =0.2cm,gray,thin] (0,0) grid (2cm,2cm);
		\draw [blue, ultra thick] (0.2*3,0.2*3)--(0.2*3,0.2*7)--(0.2*7,0.2*7)--(0.2*7,0.2*3)--(0.2*3, 0.2*3);
		\draw [blue,ultra thick] (0.2*2,0.2*4)--(0.2*3,0.2*4);
		\draw [blue,ultra thick] (0.2*2,0.2*5)--(0.2*3,0.2*5);
		\draw [blue,ultra thick] (0.2*2,0.2*6)--(0.2*3,0.2*6);
		\draw [blue,ultra thick] (0.2*2,0.2*4)--(0.2*3,0.2*4);
		\draw [blue,ultra thick] (0.2*4,0.2*2)--(0.2*4,0.2*4)--(0.2*3, 0.2*4);
		\draw [blue,ultra thick] (0.2*4,0.2*8)--(0.2*4,0.2*6)--(0.2*3, 0.2*6);
		\draw [blue,ultra thick] (0.2*6,0.2*2)--(0.2*6,0.2*4)--(0.2*8, 0.2*4);
		\draw [blue,ultra thick] (0.2*6,0.2*8)--(0.2*6,0.2*6)--(0.2*8, 0.2*6);
		\draw [blue,ultra thick] (0.2*5,0.2*2)--(0.2*5,0.2*3);
		\draw [blue,ultra thick] (0.2*7,0.2*5)--(0.2*8,0.2*5);
		\draw [blue,ultra thick] (0.2*5,0.2*7)--(0.2*5,0.2*8);
		\draw [blue,ultra thick] (0.2*1,0.2*5)--(0.2*2,0.2*5);
		\draw [blue,ultra thick] (0.2*2,0.2*4)--(0.2*2,0.2*6);
		\draw [blue,ultra thick] (0.2*8,0.2*5)--(0.2*9,0.2*5);
		\draw [blue,ultra thick] (0.2*8,0.2*4)--(0.2*8,0.2*6);
		\draw [blue,ultra thick] (0.2*2,0.2*7)--(0.2*3,0.2*7)--(0.2*3,0.2*8);
		\draw [blue,ultra thick] (0.2*7,0.2*8)--(0.2*7,0.2*7)--(0.2*8,0.2*7);
		\draw [blue,ultra thick] (0.2*2,0.2*3)--(0.2*3,0.2*3)--(0.2*3,0.2*2);
		\draw [blue,ultra thick] (0.2*8,0.2*3)--(0.2*7,0.2*3)--(0.2*7,0.2*2);
	\end{tikzpicture}
	\caption{$\Ct^m_h:$ the squares colored in red and yellow;\; $\Ct^m_{h,B}:$ the squares colored in red;\; $\Ce^m_{h,B}:$ the edges colored in blue;\; $D^m:$ the open domain colored in red and yellow; $\Gamma^m_h$: the black circle; $\partial\tilde\Omega^m_h$: the yellow circle.}
	\label{fig:Tnh}
\end{figure}
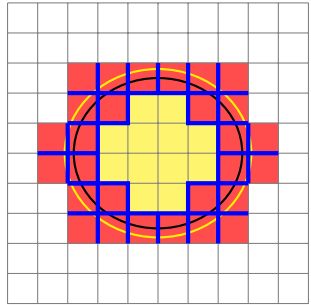

Now we define the finite element spaces as follows
\begin{align*}
V(k,\Ct_h) := \big\{v\in\Hone[D]:
v|_K\in Q_k(K),\;\forall\,K\in \Ct_h\big\}, \quad
V(k,\Ct^m_h) := \big\{v|_{D^m}:	v\in V(k,\Ct_h) \big\} ,
\end{align*}
where $Q_k$ is the space of polynomials whose degrees are no more than $k$ for each variable.
The space of piecewise regular functions over the mesh $\Ct^m_h$ is defined by
\ben
H^j(\Ct^m_h) := \big\{v\in \Ltwo[D^n]:\;
v|_K\in H^j(K),\;\forall\, K\in\Ct^m_h\big\},\qquad j\ge 1.
\een
It is clear that $V(k,\Ct_h^m)\subset H^j(\Ct^m_h)$.
We will use the notations $\BV(k,\Ct_h)=\big(V(k,\Ct_h)\big)^2$ and
$\BV(k,\Ct^m_h)=\big(V(k,\Ct^m_h)\big)^2$ in the rest of the paper.

\subsection{Forward flow map $\BX_h^{n-1,n}$}
\label{sec:Xni}

Using the one-step maps $\BX^{m,m-1}_h$ with $0\le m<n$, we can define the multi-step backward flow map ${\BX}_h^{n-1,n-i}$: $\bar\Omega^{n-1}_h\to \bar\Omega^{n-i}_h$ as
\begin{equation}\label{Xmstep}
{\BX}_h^{n-1,n-i} :={\BX}_h^{n-i+1,n-i}\circ{\BX}_h^{n-i+2,n-i+1}\circ \cdots\circ {\BX}_h^{n-1,n-2},\qquad 1\le i\le k.
\end{equation}
Similar to \eqref{eq:DnX}, we first define the forward flow map at $t_n$
\begin{equation}\label{hXnh}
\BX^{n-1,n}_h(\Bx)
:=\frac{1}{a_0^k}\sum_{i=1}^k
\big[\tau b_i^k\Bu_h^{n-i}\circ\BX^{n-1,n-i}_h(\Bx)-a_i^k\BX^{n-1,n-i}_h(\Bx)\big],
\quad \forall\,\Bx\in\bar\Omega^{n-1}_h.
\end{equation}
Since $\BX^{n-1,n}_h$ represents the forward evolution of computational domain from $t_{n-1}$ to $t_n$,
we call it the {\it forward flow map}.

\subsection{The computational domain $\Omega^n_h$}
\label{sec:domain}

Next we present the surface-tracking algorithm which generates the approximate boundary $\Gamma^n_h$, or equivalently, the computational domain $\Omega^n_h$.
In \cite{zha18}, Zhang and Fogelson proposed a surface-tracking algorithm which uses cubic spline interpolation and explicit expression of driving velocity.
Here we present a modified algorithm which uses the numerical solution as the driving velocity.

Let $\Cp^0=\left\{\Bp^0_j: 0\le j\le J^0\right\}$ be the set of control points on the initial boundary  $\Gamma^0_h:=\Gamma_0$. Suppose that the arc length of $\Gamma^0_h$ between $\Bp^0_0$ and $\Bp^0_{j}$ equals to $L^0_j=j\eta$ for $1\le j\le J^0$, where $\eta:= L^0/J^0$ and $L^0$ is the arc length of $\Gamma^0_h$. For all $0\le m<n$, suppose that we are given with the set of control points $\Cp^m=\{\Bp^m_j: 0\le j\le J^m\}\subset\Gamma^m_h$ and the parametric representation $\chibf_m$ of $\Gamma^m_h$, which satisfies
\begin{equation*}
\chibf_m(L^m_j) = \Bp^m_j, \qquad
L^m_j = \sum_{i=0}^j\SN{\Bp^m_{i+1}-\Bp^m_{i}},\quad 0\le i\le J^m.
\end{equation*}

\begin{algorithm}\label{alg:mars}
{\sf
Given $n\ge 1$ and a constant $\delta\in (0,0.5]$, the surface-tracking algorithm for constructing $\Gamma^n_h$ consists of three steps.
\begin{enumerate}[leftmargin=6mm]
\item Trace forward each control point in $\Cp^{n-1}$ to obtain the new set of control points $\Cp^n= \{\Bp^{n}_j: j=0,\cdots,J^n\}$, where $\Bp^{n}_j=\BX^{n-1,n}_h(\Bp^{n-1}_j)$ and $J^n = J^{n-1}$.

\item Adjust $\Cp^n$. For each $0\le j< J^n$, let $M_j$ be the smallest integer no less than $|\Bp_{j+1}^n-\Bp^n_{j}|/\eta$.
\begin{itemize}[leftmargin=4mm]
\item If $M_j>1$, define $\Delta l_j := (L^{n-1}_{j+1}-L^{n-1}_{j})/M_j$ and update $\Cp^n$, $J^n$ as follows
\begin{equation}\label{alg-step2}
\Cp^n \leftarrow  \Cp^n \cup\big\{
\Bp^{n}_{j,m}: 1\le m < M_j \big\}, \qquad
J^n \leftarrow   J^n + M_j-1 .
\end{equation}
where $\Bp^{n}_{j,m} = \BX_h^{n-1,n}(\Bp^{n-1}_{j,m})$ and $\Bp^{n-1}_{j,m} = \chibf_{n-1}(L^{n-1}_{j}+m\Delta l_j)$.

\item Otherwise, remove control points from $\Cp^n$ as many as possible such that $\delta\eta < \SN{\Bp^n_{j+1}-\Bp^n_j}\le \eta$ holds for all $j$.
\end{itemize}
		
\item Based on the point set $\Cp^n$ and the nodal set
$\Cl^n= \big\{\sum_{j=0}^i\big|\Bp^n_{j+1}-\Bp^n_{j}\big|:0\le i< J^n\big\}$, we construct the cubic spline function
$\chibf_{n}$ and define
\begin{equation*}
\Gamma^{n}_h:=\left\{\chibf_{n}(l): l\in [0,L^{n}]\right\}, \qquad
L^n:= \sum_{j=0}^{J_n}\big|\Bp^n_{j+1}-\Bp^n_{j}\big|.
\end{equation*}
\end{enumerate}}
\end{algorithm}

\begin{remark}
Step~2 of Algorithm~\ref{alg:mars} not only adapts the interface-tracking algorithm to severely deforming domains, but also enhances the stability of the cubic spline interpolation. For low-speed flow and short-time simulations, it is also reasonable to remove Step~2 from Algorithm~\ref{alg:mars} and set $\Cl^n\equiv\Cl^0$.
\end{remark}

\subsection{Backward flow map $\BX^{n,n-1}_h$}
\label{sec:bfm}

Note that $\Omega^n_h$ is constructed with cubic spline interpolation over the point set $\Cp^n$. Generally we have $\Omega^n_h\ne \BX^{n-1,n}(\bar\Omega^{n-1}_h)$. Moreover, the computation of $\big(\BX^{n-1,n}_h\big)^{-1}$ is very time-consuming in practical computations. The backward flow map $\BX^{n,n-1}_h$ is an approximation of $(\BX^{n-1,n}_h)^{-1}$ and will be defined in two steps.

{\bf Step~1. Define an approximation $\tilde\BX^{n-1,n}_h$ of $\BX^{n-1,n}_h$.} We shall define
$\tilde\BX^{n-1,n}_h$: $\bar\Omega^{n-1}_h\to\bbR^2$ piecewise on each element $K\in\Ct^{n-1}_{h}$.

\begin{figure}[http!]
\begin{center}
\includegraphics[width=0.25\textwidth]{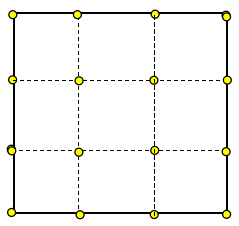}
\caption{Uniform nodal points on an interior element $K\in\Ct^{n-1}_{h,I}$.}
\label{fig:XnK}
\end{center}
\end{figure}

First we consider each interior element $K\in\Ct^{n-1}_{h,I}$. Let $\BA^K_{ij}$, $0\le i,j\le k$, be the nodal points taken uniformly on $K$ (see Fig.~\ref{fig:XnK}).
Let $P_k([0,1])$ be the space of polynomials on $[0,1]$ with degrees $\le k$ and let $\{b_0,\cdots,b_k\}$ be the basis of $P_k([0,1])$ satisfying
\ben
b_i(j/k)=\delta_{i,j},\qquad 0\le j\le k,
\een
where $\delta_{i,j}$ stands for the Kronecker delta function.
An isoparametric transform from the reference element $\hat{K}=[0,1]^2$ to $K$ is defined as
\begin{equation}\label{FK}
F_K(\xibf) := \sum_{i,j=0}^k\BA^K_{ij}b_i(\xi_1)b_j(\xi_2),\qquad
\forall\,\xibf=(\xi_1,\xi_2)\in\hat{K}.
\end{equation}
We use $\BX^{n-1,n}_h$ to trace forward each $\BA^K_{ij}$ from $t_{n-1}$ to $t_n$ and get the nodal points
\ben
\BX^{n-1,n}_h\big(\BA^K_{ij}\big),\qquad 0\le i,j\le k.
\een
They define another isoparametric transform
\begin{equation}\label{GK}
G_K(\xibf) := \sum_{i,j=0}^k\BX^{n-1,n}_h\big(\BA^K_{ij}\big) b_i(\xi_1)b_j(\xi_2)
\qquad \forall\,\xibf=(\xi_1,\xi_2)\in\hat{K}.
\end{equation}
We get a homeomorphism from $K$ to $K^n:=\big\{G_K(\xibf): \xibf\in\hat{K}\big\}$
\begin{equation}\label{X-iK}
\tilde\BX^{n-1,n}_K :=G_K\circ F_K^{-1} .
\end{equation}

\begin{figure}[http!]
\begin{center}
\includegraphics[width=0.25\textwidth]{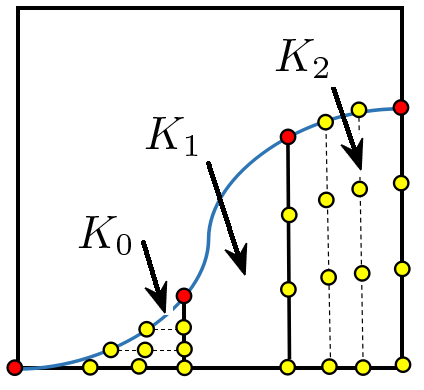}
\caption{The partition of $K\cap\Omega^{n-1}_h$ into curved polygons $K_0$, $K_1$, and $K_2$, where red dots stand for control points in $\Cp^{n-1}$. Quasi-uniform nodal points are shown on $K_0$ and $K_2$, respectively.}
\label{fig:XnKm}
\end{center}
\end{figure}

Next we consider each $K\in\Ct^{n-1}_{h,B}$. Since $\Gamma^{n-1}_h\cap K$ is represented by the piecewise cubic funtion $\chibf_{n-1}$. Let $M$ be the number of control points in the interior of $K$. We subdivide $K\cap\bar\Omega^{n-1}_h$ into $M+1$ curved polygons
\begin{equation*}
K\cap\bar\Omega^{n-1}_h=K_0\cup K_1\cup\cdots\cup K_M,\qquad
\mathring{K}_l\cap\mathring{K}_m=\emptyset\quad
\hbox{for}\;\; l\ne m.
\end{equation*}
Each $K_m$ is either a curved triangle or a curved quadrilateral with only one curved edge $\partial K_m\cap\Gamma^{n-1}_h$ (see Fig.~\ref{fig:XnKm}). We define isoparametric transforms for the two cases of $K_m$ respectively.
\vspace{1mm}
\begin{itemize}[leftmargin=5mm]
\item If $K_m$ is a curved quadrilateral (see $K_2$ in Fig.~\ref{fig:XnKm}),
we take nodal points
$\BA^{K_m}_{ij}, 0\le i,j\le k$, quasi-uniformly on $K_m$. Similar to \eqref{FK}--\eqref{GK}, we define two isoparametric transforms as
\begin{align}\label{FmGm-K}
F_{K_m}(\xibf) := \sum_{i,j=0}^k
\BA^{K_m}_{ij}b_i(\xi_1)b_j(\xi_2),\quad
G_{K_m}(\xibf) := \sum_{i,j=0}^k\BX^{n-1,n}_h\big(\BA^{K_m}_{ij}\big) b_i(\xi_1)b_j(\xi_2),\quad
\forall\,\xibf\in\hat{K}.
\end{align}

\item If $K_m$ is a curved triangle (see $K_0$ in Fig.~\ref{fig:XnKm}),
we take $(k+1)(k+2)/2$ nodal points $\BA^{K_m}_{ij},0\le i+j\le k$, quasi-uniformly on $K_m$. Let
$\hat{T}$ be the reference triangle with vertices $(0,0)$, $(1,0)$, and $(0,1)$.
The two isoparametric transforms are defined as
\begin{align}\label{FmGm-T}
F_{K_m}(\xibf) := \sum_{i+j=0}^k
\BA^{K_m}_{ij}b_{ij}(\xibf),\quad
G_{K_m}(\xibf) := \sum_{i+j=0}^k\BX^{n-1,n}_h
\big(\BA^{K_m}_{ij}\big) b_{ij}(\xibf),\quad
\forall\,\xibf\in\hat{T},
\end{align}
where $b_{ij}\in P_k(\hat{T})$ satisfies $b_{ij}(l/k,m/k)=\delta_{i,l}\delta_{j,m}$ for two integers satisfying $0\le l+m\le k$.
\end{itemize}
\vspace{1mm}
In both cases, they define a homeomorphism from $K_m$ to $K^n_m :=\big\{G_{K_m}(\xibf):\xibf\in\hat{K}\big\}$
\begin{equation}\label{Kmn}
\tilde\BX^{n-1,n}_{K_m} :=G_{K_m}\circ F_{K_m}^{-1}.
\end{equation}
Therefore, we obtain a homeomorphism $\tilde\BX^{n-1,n}_K$: $K\cap\bar\Omega^{n-1}_h \to K^n:=\cup_{m=0}^M K^n_m$
\begin{equation}\label{X-bK}
\tilde\BX^{n-1,n}_K\big|_{K_m} = \tilde\BX^{n-1,n}_{K_m} .
\end{equation}
Define $\Omega^n_{\BX}:=\cup_{K\in\Ct^{n-1}_h}K^n$. Combining \eqref{X-iK} and \eqref{X-bK}, we obtain a homeomorphism $\tilde\BX^{n-1,n}_h$: $\bar\Omega^{n-1}_h\to \bar\Omega^n_{\BX}$ which is defined piecewise as follows
\begin{equation}\label{tXn}
\tilde\BX^{n-1,n}_h = \tilde\BX^{n-1,n}_K \quad \hbox{on}\;\; K\cap\bar\Omega^{n-1}_h.
\end{equation}

{\bf Step~2. Define the backward flow map $\BX^{n,n-1}_h$.} First we let
$\tilde\BX^{n,n-1}_h :=\big(\tilde\BX^{n,n-1}_h\big)^{-1}$ which is a homeomorphism from $\bar\Omega^n_{\BX}$ to $\bar\Omega^{n-1}_h$. Note that $\tilde\BX^{n,n-1}_h$ is undefined on $\Omega^n_h\backslash\Omega^n_{\BX}$.
Next we extend it to $\bar\Omega^n_h\cup\bar\Omega^n_{\BX}$ and denote the extension by the same notation.

For any $\Bx\in\partial\Omega^n_{\BX}$, let $\Bn(\Bx)$ be the unit outer normal at $\Bx\in\partial \Omega^n_{\BX}$ (see Fig.~\ref{fig:ext-Xnh}). The extension is defined by constant along $\Bn(\Bx)$, namely,
\begin{equation}\label{nx}
\tilde\BX^{n,n-1}_h\big(\Bx+t\Bn(\Bx)\big) \equiv \tilde\BX^{n,n-1}_h(\Bx)
\qquad \hbox{for all}\;\; t>0 \;\;\hbox{satisfying}\; \;
\Bx+t\Bn_\Bx\in \bar{\Omega}^n_h.
\end{equation}
Finally, we define the backward flow map as
\begin{equation}\label{bfm}
\BX^{n,n-1}_h = \tilde\BX^{n,n-1}_h\big|_{\bar\Omega^n_h}.
\end{equation}
Clearly $\BX^{n,n-1}_h$: $\bar{\Omega}^n_h\to \bar{\Omega}^{n-1}_h$ is neither surjective nor injective.

\begin{figure}[http!]
\begin{center}
\includegraphics[width=0.25\textwidth]{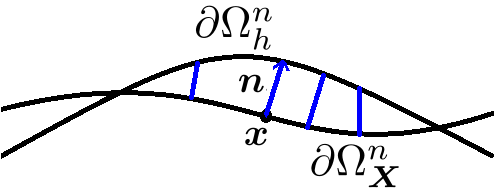}
\caption{Constant extension of $\tilde{\BX}_h^{n,n-1}$ to $\bar{\Omega}^n_h\backslash \bar\Omega^{n}_{\BX}$ along the normal direction $\Bn=\Bn(\Bx)$.}
\label{fig:ext-Xnh}
\end{center}
\end{figure}

\begin{remark}
By \eqref{FK}--\eqref{Kmn}, the computation of $\tilde\BX^{n-1,n}_h = \big(\tilde\BX^{n-1,n}_h\big)^{-1}$ requires two operations
\begin{itemize}[leftmargin=5mm]
\item tracing the nodal points $\BA^K_{ij}$, $\BA^{K_m}_{ij}$ one step forward $($see \eqref{GK} and \eqref{FmGm-K}--\eqref{FmGm-T}$)$ ,

\item computing composite isoparametric transforms $F_K\circ G_K^{-1}$ and $F_{K_m}\circ G_{K_m}^{-1}$ $($see \eqref{X-iK} and \eqref{Kmn}$)$ .
\end{itemize}
In view of \eqref{bfm}, we know that the computational complexity of $\BX^{n,n-1}_h$ is much more economic than that of the inverse map $\big(\BX^{n-1,n}_h\big)^{-1}$. Nevertheless, optimal convergence orders are observed in numerical experiments which use $\BX^{n,n-1}_h$.
\end{remark}

\subsection{Finite element scheme for computing $\Bu^n_h$}

Remember that we have obtained the computation domain $\Omega^n_h$ in subsection~\ref{sec:domain}. Let $\Ct^n_h$ be the sub-mesh covering $\Omega^n_h$ and let $\Ce_{h,B}^{n}$ be the set of boundary-zone edges.

For any edge $E\in\Ce_{h,B}^{n}$, suppose $E=K_1\cap K_2$ with $K_1,K_2\in\Ct_h^n$. Let $\Bn_{K_1}$ be the unit normal of $E$ , pointing to the exterior of $K_1$, and let $\Bn_{K_2}=-\Bn_{K_1}$. The normal jump of a scalar function $v$ across $E$ is defined by
\ben
\jump{v}(\Bx) =\lim_{\varepsilon\to 0+}
\left[v(\Bx-\varepsilon\Bn_{K_1}) \Bn_{K_1}
+v(\Bx-\varepsilon\Bn_{K_2})\Bn_{K_2}\right]\qquad \forall\,\Bx\in E.
\een
Clearly $\jump{\Bv}$ is a matrix function on $E$ if $\Bv$ is a vector function. For two vector functions $\Bv$ and $\Bw$, the Euclidean inner product of $\jump{\Bv}$ and $\jump{\Bw}$ is defined by
\ben
\jump{\Bv}:\jump{\Bw} =\jump{v_1}\cdot \jump{w_1} +
\jump{v_2}\cdot \jump{w_2}.
\een
Furthermore, we define two bilinear forms on $\BH^{k+1}(\Ct^n_h) \cap\Honev[D^n]$ as follows
\begin{align}	
\mathscr{A}^n_h(\Bw,\Bv):=\,&
\int_{\Omega^n_{h}} \nu \nabla \Bw\cdot\nabla \Bv
	+ \mathscr{J}_{h}^{n}(\Bw,\Bv) ,\label{A-nh}\\
	\mathscr{J}^n_h(\Bw,\Bv):=\,& \gamma
	\sum_{E\in \Ce_{h,B}^{n}}
	\sum_{l=1}^k \frac{h^{2l-1}}{[(l-1)!]^2}\int_E \nu
	\jump{\partial_{\Bn}^l \Bw} :
	\jump{\partial_{\Bn}^l \Bv}, \label{J-nh}
\end{align}	
where $\gamma $ is a positive constant whose value will be specified in the section for numerical experiments.
Here $\mathscr{J}_{h}^{n}$, called ``ghost penalty stabilization'' in the literature (cf. \cite{bur10-1}), is used to enhance the stability of numerical solutions.

Define $\ul{\BU^n_h}=\big[\BU^{n-k,n}_h,\cdots,\BU^{n,n}_h\big]^\top$, $\BU^{n,n}_h=\Bu^n_h$, and $\BU_h^{n-i,n}:=\Bu_h^{n-i}\circ{\BX}_h^{n,n-i}$ for $1\le i\le k$. Similar to \eqref{un-eqn}, we define the discrete BDF-$k$ difference operator as
\begin{equation}\label{def:Lambdak}
\frac{1}{\tau}\Lambda^k \ul{\BU_h^n}
:= \frac{1}{\tau}\sum_{i=0}^k a_i^k \BU_h^{n-i,n}.
\end{equation}
\begin{center}
\fbox{\parbox{0.975\textwidth}
{The finite element approximation to problem \eqref{un-pro} is to seek  $\Bu^n_h\in \BV(k,\Ct_h^{n})$ such that
\begin{align}	\label{eq:fully-discrete}
\big(\Lambda^k \ul{\BU^n_h}, \Bv_h\big)_{\Omega^n_{h}}
	+\tau  \mathscr{A}^n_h(\Bu_h^n,\Bv_h)
	=\tau (\Bf^n,\Bv_h)_{\Omega^n_{h}} \qquad
	\forall\,\Bv_h\in \BV(k,\Ct_h^{n}).
\end{align}}}
\end{center}

\begin{remark}
Suppose that the approximation of $\Gamma^n_h$ to the exact boundary $\Gamma_{t_n}$ is high-order, say
\ben
\max_{\Bx\in\Gamma_{t_n}}\min_{\By\in\Gamma^n_h} \SN{\Bx-\By}  = O(h^k), \qquad k\ge 2.
\een
From \eqref{domain-tm} we know that $\bar\Omega_{t_n}\subset\tilde\Omega^n_h$ if $h$ is small enough. This guarantees that the numerical solution $\Bu^n_h$ is well-defined on the exact domain $\bar\Omega_{t_n}$. That is why we choose an enlarged domain $\tilde\Omega^n_h\subset D^n$ to define the finite element space, instead of the tracked domain $\Omega^n_h$.

\end{remark}

To end this section, we prove the well-posedness of the discrete problem.

\begin{theorem}
Suppose that the pre-calculated solutions $\Bu^0_h,\cdots,\Bu^{k-1}_h$ are given. Then the discrete problem \eqref{eq:fully-discrete} has a unique solution
$\Bu^n_h\in \BV(k,\Ct^n_h)$ for each $k\le n\le N$.
\end{theorem}
\begin{proof}
From \eqref{A-nh} and \eqref{J-nh}, it is easy to see that the bilinear form $\mathscr{A}^n_h$ is symmetic and semi-positive on $\BV(k,\Ct^n_h)$. It suffices to show that
$a_0^k\tau^{-1}(\Bv_h,\Bv_h)_{\Omega^n_h}+\mathscr{A}^n_h(\Bv_h,\Bv_h)=0$ implies $\Bv_h\equiv 0$.
It clearly yields
\ben
\Bv_h=0 \quad \hbox{in}\;\;\Omega^n_h,\qquad
\jump{\partial_{\Bn}^l \Bv_h}=0\quad \hbox{on any}\;E\in\Ce^n_{h,B}, \quad 0\le l\le k.
\een
The conclusion is obvious.
\end{proof}

\section{Numerical experiments}

In this section, we use three numerical examples to verify the convergence orders of the proposed finite element method. The second and third examples demonstrate the robustness of the method for severely deforming domains.

\subsection{An efficient algorithm for computing $\big(\BU_h^{n,n-i}, \Bv_h\big)_{\Omega_h^n}$}
\label{sec:Uni}

Remember from \eqref{eq:fully-discrete} that we need to calculate the integrals accurately and efficiently
\ben
\big(\BU_h^{n,n-i},\Bv_h\big)_{\Omega^n_h}
=\big(\Bu_h^{n-i}\circ\BX^{n,n-i}_h,\Bv_h\big)_{\Omega^n_h} \quad\hbox{for}\;\;
1\le i\le k\;\;\hbox{and}\;\;\Bv_h\in\BV(k,\Ct^n_h).
\een
For $i\ge 2$, the computation of $\big(\Bu_h^{n-i}\circ\BX^{n,n-i}_h,\Bv_h\big)_{\Omega^n_h}$ involves multi-step map $\BX^{n,n-i}_h$ and is time-consuming.
To simplify the computation while keep accuracy, we make the replacement in calculating the integral
\ben
\big(\BU_h^{n,n-i},\Bv_h\big)_{\Omega^n_h}
\approx \big(\hat\Bu_h^{n,n-i},\Bv_h\big)_{\Omega^n_h},
\een
where $\hat\Bu_h^{n,n-i}\in \BV(k,\Ct_h^n)$ is defined in Algorithm~\ref{alg:int} via modified $L^2$-projections. The computation of $\hat\Bu_h^{n,n-i}$ only involves one-step maps and reduces the computational time significantly. \vspace{1mm}

\begin{algorithm}\label{alg:int}
The functions $\hat\Bu_h^{n,n-i}\in \BV(k,\Ct_h^n)$, $1\le i\le k$, are calculated successively.
Suppose that $\hat\Bu_h^{n-1,n-i}\in \BV(k,\Ct_h^{n-1})$, $2\le i\le k$, have been obtained.
Find $\hat\Bu_h^{n,n-i}\in \BV(k,\Ct_h^n)$ such that
\begin{equation}\label{L2pro-1}
\mathscr{M}^n_h\big(\hat\Bu_h^{n,n-i},\Bv_h\big) = (\hat{\Bu}_h^{n-1,n-i}\circ {\BX}_h^{n,n-1},\Bv_h)_{\Omega_h^n}\quad
\forall\,\Bv_h\in \BV(k,\Ct_h^n), \quad 1\le i\le k.
\end{equation}
where $\hat{\Bu}_h^{n-1,n-1}=\Bu^{n-1}_h$ and
\begin{align*}	
\mathscr{M}^n_h(\Bw,\Bv):=\,& \int_{\Omega^n_{h}} \Bw\cdot\Bv
	+\gamma \sum_{E\in \Ce_{h,B}^n}
	\sum_{l=1}^k \frac{h^{2l+1}}{(l!)^2}\int_E
	\jump{\partial_{\Bn}^l \Bw}:
	\jump{\partial_{\Bn}^l \Bv}.
\end{align*}	
\end{algorithm}

\subsection{The computation of $\big(\Bw_h\circ\BX_h^{n,n-1}, \Bv_h\big)_{\Omega_h^n}$
for any $\Bw_h\in \BV(k,\Ct_h^{n-1})$ and $\Bv_h\in \BV(k,\Ct_h^{n})$}

Using \eqref{bfm}, we approximate the integral as follows
\ben
\big(\Bw_h\circ {\BX}_h^{n,n-1},\Bv_h\big)_{\Omega^{n}_h}
=\big(\Bw_h\circ \tilde{\BX}_h^{n,n-1},\Bv_h\big)_{\Omega^{n}_h}
\approx \big(\Bw_h\circ \tilde{\BX}_h^{n,n-1},\Bv_h\big)_{\Omega^{n}_\BX}
=\big(\Bw_h\circ \big(\tilde{\BX}_h^{n-1,n}\big)^{-1},\Bv_h\big)_{\Omega^{n}_\BX},
\een
where $\Omega^{n}_\BX =\tilde\BX^{n-1,n}_h(\Omega^{n-1}_h)$. By \eqref{tXn},
$\tilde\BX^{n-1,n}_h$ is a homeomorphism from $\bar\Omega^{n-1}_h$ to $\Omega^{n}_\BX$
and provides an approximation to the forward flow map $\BX^{n-1,n}_h$. Moreover, $\tilde\BX^{n-1,n}_h$ is defined piecewise on each $K\in\Ct^{n-1}_h$.
By \eqref{FK}--\eqref{Kmn}, the integral on the right-hand side can be calculated as follows
\begin{align}
\big(\Bw_h\circ \big(\tilde{\BX}_h^{n-1,n}\big)^{-1},\Bv_h\big)_{\Omega^{n}_\BX}
=\,& \sum_{K\in\Ct^{n-1}_{h,I}}
\int_{\hat{K}}(\Bw_h\circ F_K)\cdot (\Bv_h\circ G_K) \SN{\det(\D_{\xibf} G_K)} \notag\\
&+\sum_{K\in\Ct^{n-1}_{h,B}}
\sum_{m=0}^M \int_{\hat K_m}(\Bw_h\circ F_{K_m})\cdot
(\Bv_h\circ G_{K_m})\SN{\det(\D_{\xibf} G_{K_m})}, \label{cal-wv}
\end{align}
where $\hat K_m=\hat K$ if $K_m$ is a curved quadriliteral and $\hat K_m =\hat T$ if $K_m$ is a curved triangle (see subsection~\ref{sec:bfm}).
Here $F_K$ is the isoparametric transform from the reference element $\hat{K}$ to $K$ or $K_m$, $G_K$ is the isoparametric transform from $\hat{K}$ to $K^n$ or $K^n_m$, and $\D_{\xibf} G_K$, $\D_{\xibf}G_{K_m}$ are Jacobi matrices of  $G_K$ and $G_{K_m}$,  respectively. Each integral on reference elements will be computed with Gaussian quadrature rule of the $(k+1)^{\rm th}$-order.

Since $\hat{\Bu}_h^{n-1,n-i}\in \BV(k,\Ct^{n-1}_h)$, the right-hand side of \eqref{L2pro-1} can be calculated with the formula \eqref{cal-wv}.

\subsection{Numerical examples}

In this section, we report three examples for various scenarios of the domain: rotation, vortex shear, and severe deformation. In order to test convergence orders of the method,
we take $\Bu$ as smooth functions and set the right-hand side of \eqref{cd-model} by
\ben
\Bf= \partial_t \Bu + \Bu\cdot \nabla \Bu -\nu \Delta \Bu.
\een
The Neumann condition is set by
\ben
\Bg_N:=\nabla \Bu\cdot \Bn\quad \hbox{on}\;\;\Gamma_t.
\een

For each example, we require that the final shape of domain $\Omega_T$ coincides with the initial shape $\Omega_0$. This helps us to compute the surface-tracking error
\begin{equation}\label{eq:err of Omega}
	e_\Omega = \sum_{K\in \Ct_h}\left|\mathrm{area}\big(\Omega_T\cap K\big) - \mathrm{area}\big(\Omega^N_h\cap K\big)\right|,
\end{equation}
which is also called the geometrical error \cite{Aul03,Aul04}. Here $t_N=T$ is the final time of evolution. Since the exact boundary $\Gamma_T$ and the approximate boundary $\Gamma^N_h$ are very close, calculating the exact area of domain difference directly
\ben
e_{0,\Omega} =\mathrm{area}(\Omega_T\backslash\Omega_h^N)
    +\mathrm{area}(\Omega_h^N\backslash\Omega_T)
\een
may cause an ill-conditioned problem. In fact, it is easy to see that $e_\Omega\to e_{0,\Omega}$ as $h\to 0$.

To measure the error between the exact velocity and the numerical solution,
we use
\ben
e_0=\N{\Bu(\cdot,T)-\Bu_h^{N}}_{\BL^2(\Omega_T)},
\een
since $\Omega_T=\Omega_0$ is explicitly given. However, $\Omega_{t_n}$ are unknown in intermediate time steps. We measure the $H^1$-norm errors on the computational domains by
\begin{equation*}
e_1=\bigg(\sum_{n=k}^N \tau \SN{\Bu(\cdot,t_n)-\Bu_h^n}_{\BH^1(\Omega_h^n)}^2\bigg)^{1/2}.
\end{equation*}
Throughout this section, we set the diffusion coefficient by $\nu=1$. Other constant values of $\nu$ do not influence the convergence orders for $\tau$ small enough. The penalty parameter is set by $\gamma=0.001$ in $\mathscr{J}^n_h$. In Algorithm~\ref{alg:mars}, the parameters for surface-tracking  are set by $\eta = 0.5 h$ and $\delta = 0.01$.
\vspace{1mm}

\begin{example}[Rotation of an elliptic disk]\label{ex1}
The exact velocity is set by
\begin{equation*}
	\Bu = (0.5-y,x-0.5).
\end{equation*}
The initial domain is an elliptic disk whose center is located at $(0.5,0.5)$, major axis $R_1 = 0.6$, and minor axis $R_2=0.3$. The finial time is set by $T=\pi$.
\end{example}

The moving domain $\Omega_t$ rotates half a circle counterclockwise with angular velocity $=1$ and returns to its original shape at $t=T$.
Fig.~\ref{fig:domain-ex3} shows the snapshots of $\Omega^n_h$ which are formed with Algorithm~\ref{alg:mars} at $t_n=0$, $T/8$, $T/4$, $T/2$, $3T/4$, and $T$, respectively.
From Tables~\ref{tab:ex3-k3} and \ref{tab:ex3-k4}, we find that optimal convergence orders are obtained for both the SBDF-3 and SBDF-4 schemes, namely,
\ben
e_0 \sim \tau^k,\qquad
e_1 \sim \tau^k,\qquad
e_\Omega \sim \tau^k,\qquad k=3,4.
\een

\begin{figure}[http!]
	\centering
	\subfigure[$t=0$]{
		\includegraphics[width=4.5cm]{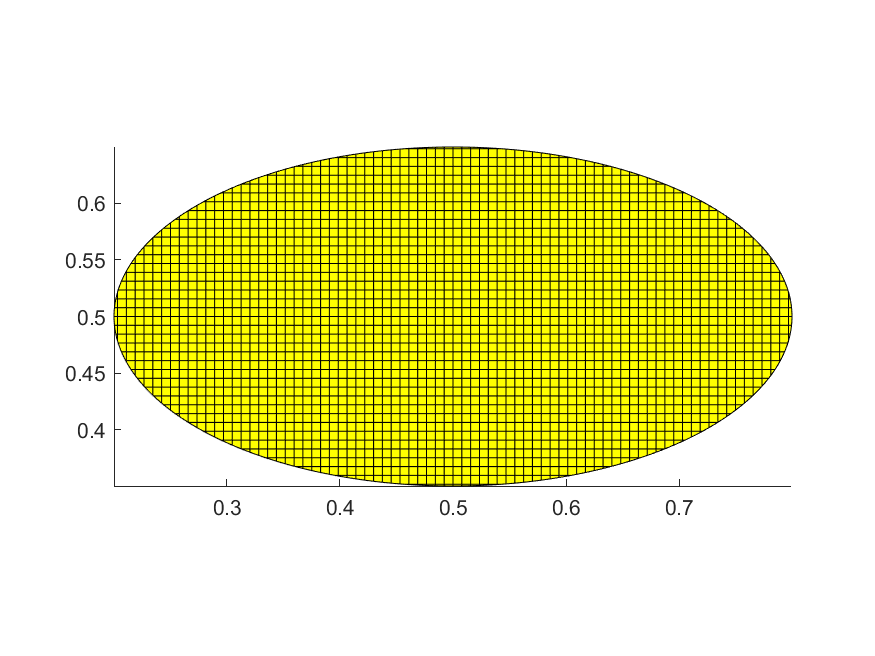}
	}
	\subfigure[$t=T/8$]{
		\includegraphics[width=4.5cm]{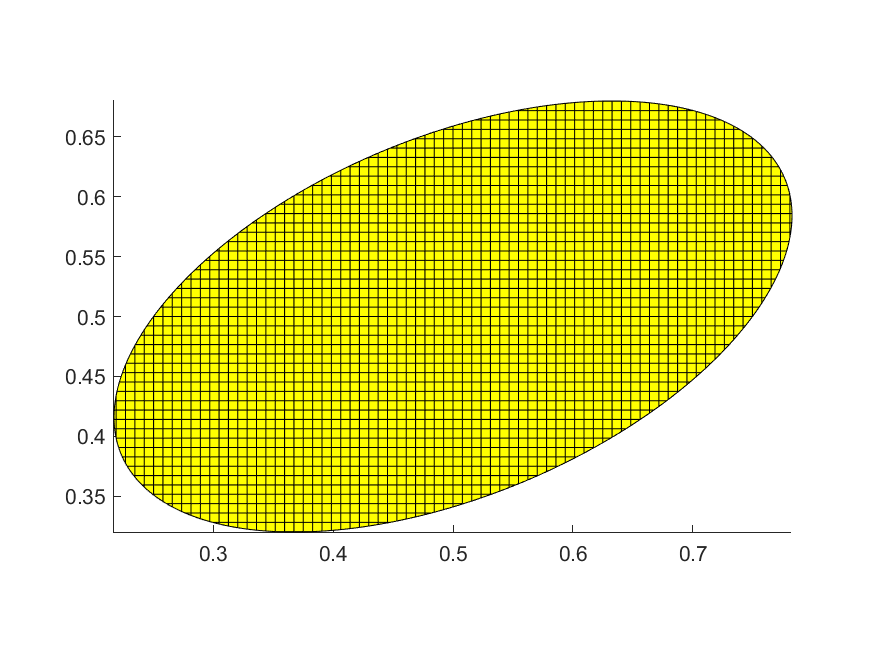}
	}
	\subfigure[$t=T/4$]{
		\includegraphics[width=4.5cm]{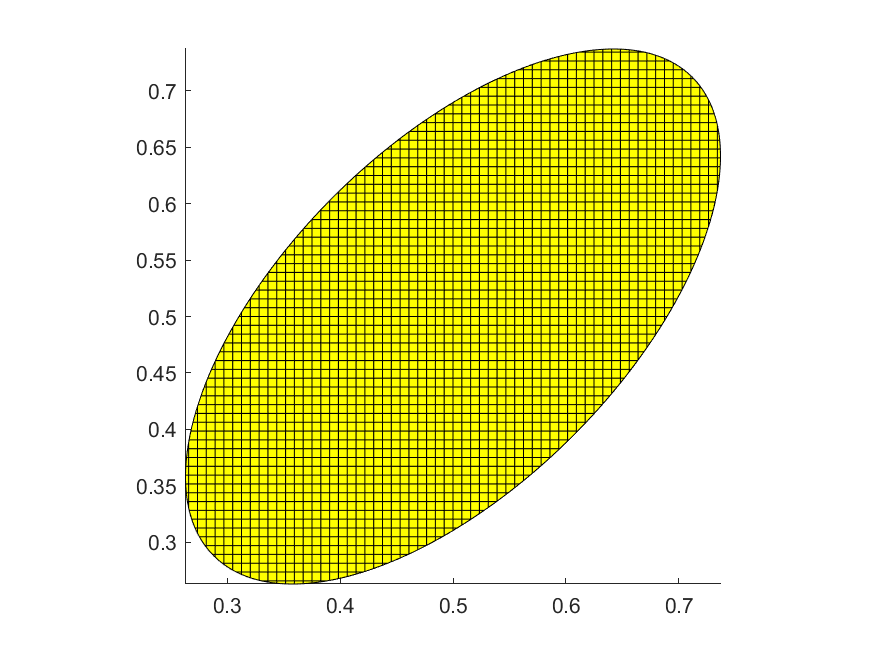}
	} \\
	\subfigure[$t=T/2$]{
		\includegraphics[width=4.5cm]{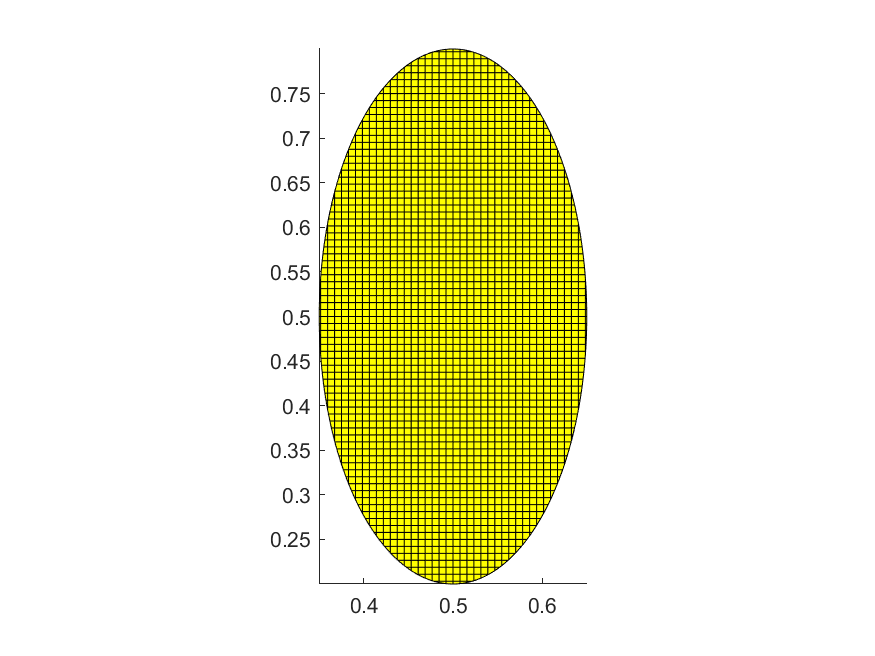}
		\label{fig3:T/2}}
	\quad
	\subfigure[$t=3T/4$]{
		\includegraphics[width=4.5cm]{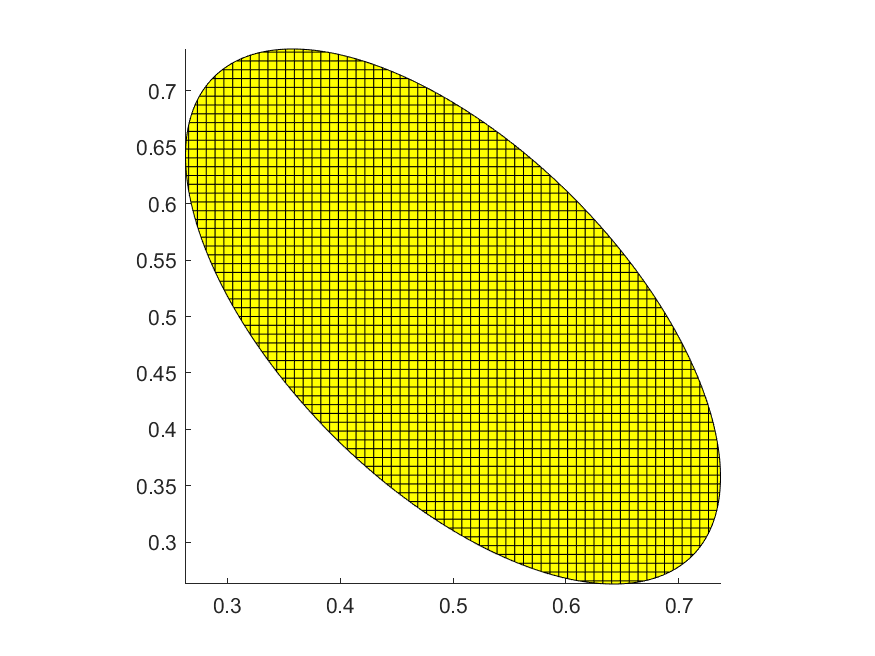}
	}
	\subfigure[t=$T$]{
		\includegraphics[width=4.5cm]{rotationN128step5.png}
	}
\caption{Computational domains $\Omega^n_h$ at different time steps
$(h=1/128,\; k=3)$.} \label{fig:domain-ex1}
\end{figure}

\begin{table}[http!]
	\center
	\caption{Convergence orders for the SBDF-3 scheme (Example~\ref{ex1}).}\label{tab:ex1-k3}
	\setlength{\tabcolsep}{3mm}
	\begin{tabular}{ |c|c|c|c|c|c|c|}
		\hline
		$h=\tau/\pi$    & $e_0$     &Order  & $e_1$     &Order & $e_\Omega$   &Order\\ \hline
	1/16  &1.06e-05  &-     &1.10e-04  &-     &4.03e-03 &-\\  \hline
	1/32  &1.36e-06  &2.93  &1.47e-05  &2.91  &5.73e-04  &2.81 \\  \hline
	1/64  &1.71e-07  &2.99  &1.89e-06  &2.96  &7.54e-05  &2.92 \\  \hline
	1/128 &2.14e-08  &3.00  &2.39e-07  &2.98  &9.64e-06  &2.97 \\  \hline
	\end{tabular}
\end{table}

\begin{table}[http!]
	\center
	\caption{Convergence orders for the SBDF-4 scheme (Example~\ref{ex1}).}\label{tab:ex1-k4}
	\setlength{\tabcolsep}{3mm}
	\begin{tabular}{ |c|c|c|c|c|c|c|}
		\hline
	$h=\tau/\pi$    & $e_0$     &Order  & $e_1$     &Order & $e_\Omega$   &Order\\ \hline
	1/16  &2.13e-06  &-  &2.17e-05  &- &4.14e-04 &-\\  \hline
	1/32   &1.34e-07  &3.98  &1.44e-06  &3.91  &2.83e-05  &3.87 \\  \hline
	1/64   &8.42e-09  &4.00  &9.26e-08  &3.96  &1.86e-06  &3.93 \\  \hline
	1/128   &5.56e-10  &3.92  &5.88e-09  &3.98  &1.20e-07  &3.96 \\  \hline
	\end{tabular}
\end{table}

\begin{example}[Vortex shear of a circular disk]\label{ex2}
The exact velocity is given by
\begin{equation*}
\Bu = \cos (\pi t/3) \big(\sin^2 (\pi x) \sin (2\pi y),\,	-\sin^2 (\pi y) \sin (2\pi x)\big).
\end{equation*}
The initial domain is a disk with radius $R=0.15$ and centering at $(0.5,0.75)$. The final time is $T=3$.
\end{example}

The domain is stretched into a snake-like region at $T/2$ and goes back to its initial shape at $t=T$.
Fig.~\ref{fig:domain-ex2} shows the snapshots of the computational domain $\Omega^n_h$ during the time evolution procedure. Tables~\ref{tab:ex1-k3} and \ref{tab:ex1-k4} show optimal convergence orders for both the SBDF-3 and SBDF-4 schemes, respectively. From Fig.~\ref{fig:domain-ex2}~(d), we find that although the domain experiences severe deformations, high-order convergence is still obtained. This demonstrates the robustness of the finite element method.

\begin{figure}[http]
	\centering
	\subfigure[t=$0$]{
		\includegraphics[width=4.5cm]{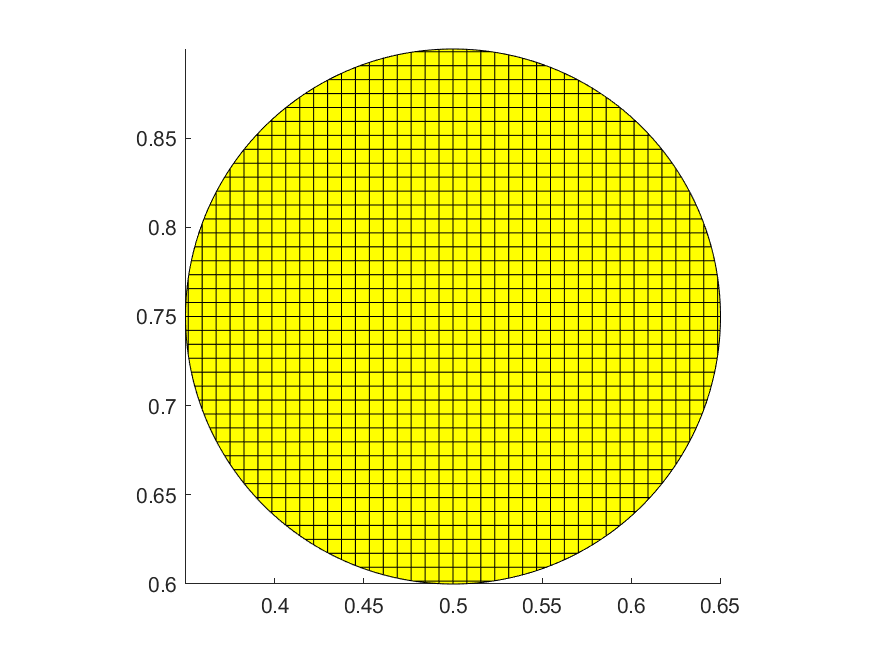}
	}
	\subfigure[$t=T/8$]{
		\includegraphics[width=4.5cm]{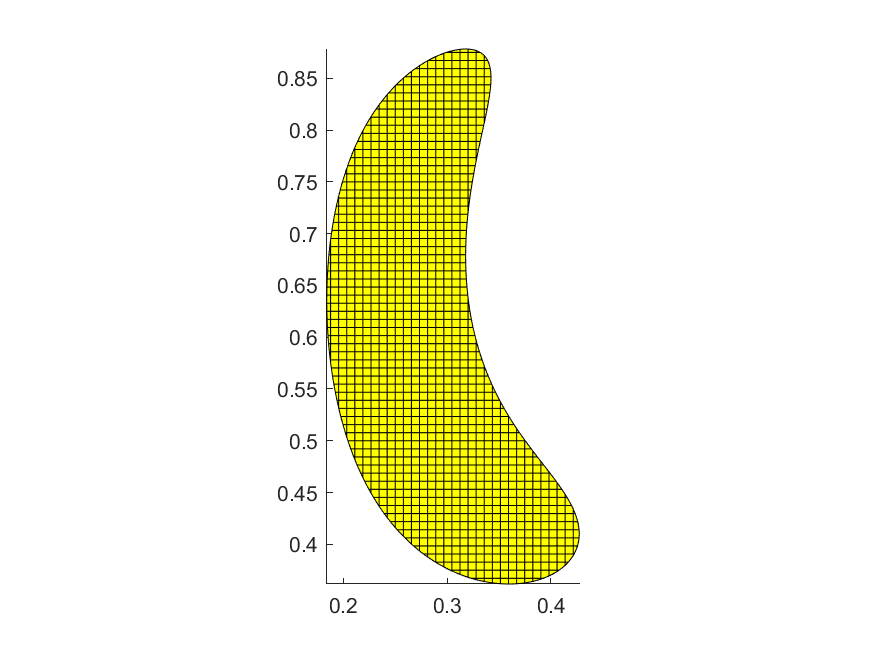}
	}
	\subfigure[$t=T/4$]{
		\includegraphics[width=4.5cm]{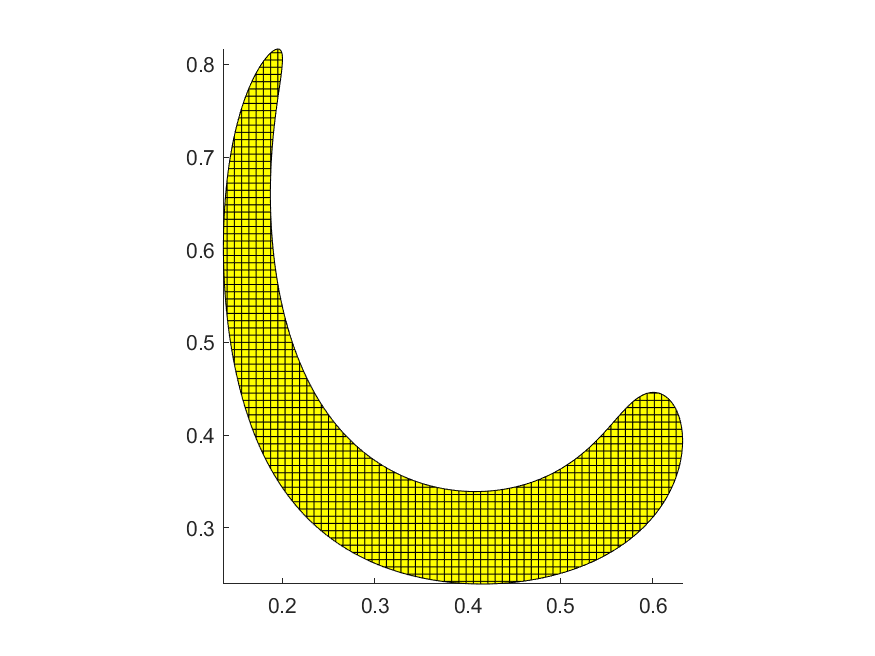}
	} \\
	\subfigure[$t=T/2$]{
		\includegraphics[width=4.5cm]{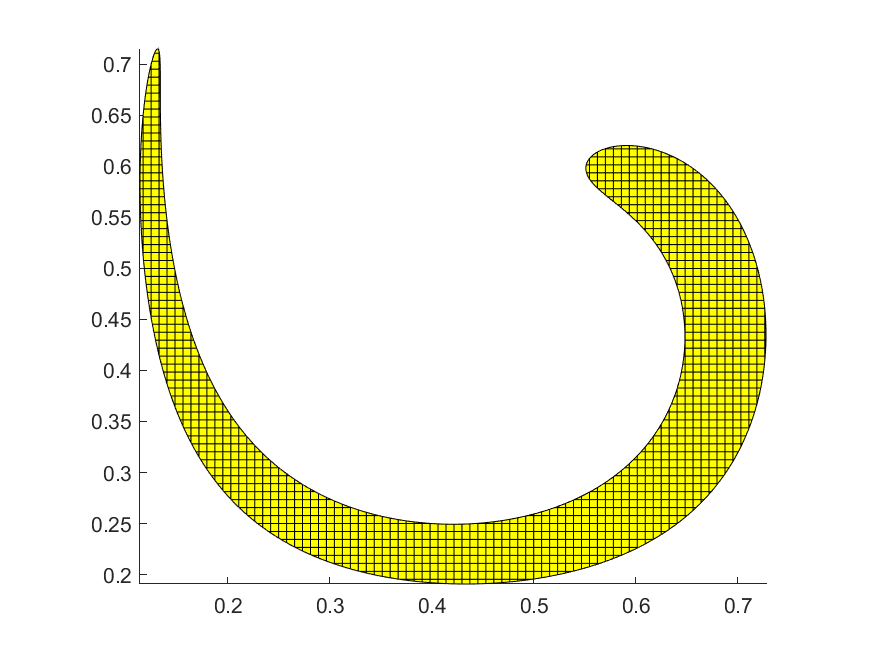}
	\label{fig1:T/2}}
	\subfigure[$t=3T/4$]{
		\includegraphics[width=4.5cm]{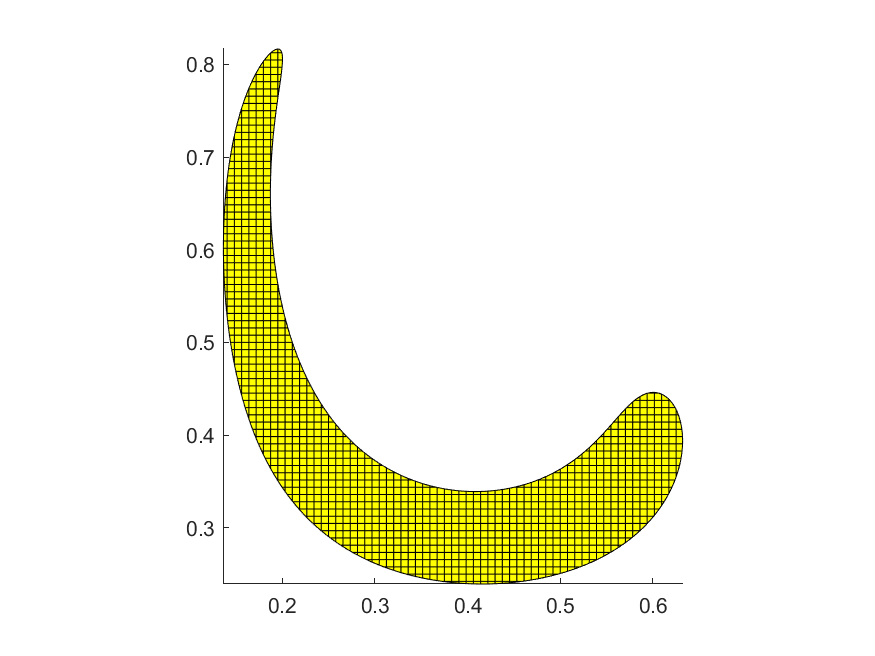}
	}
	\subfigure[$t=T$]{
		\includegraphics[width=4.5cm]{vortexstep5.png}
	}
\caption{Computational domains $\Omega^n_h$ at different time steps
$(h=1/128,\;k=3)$.} \label{fig:domain-ex2}
\end{figure}

\begin{table}[http!]
\center
\caption{Convergence orders for the SBDF-3 scheme (Example~\ref{ex2}).}\label{tab:ex2-k3}
\setlength{\tabcolsep}{3mm}
\begin{tabular}{ |c|c|c|c|c|c|c|}
\hline
$h=\tau$    & $e_0$     &Order  & $e_1$     &Order & $e_\Omega$   &Order\\ \hline
1/16  &1.17e-03  &-  &1.71e-03  &- &6.70e-03 &-\\  \hline
1/32   &1.07e-04  &3.45  &2.00e-04  &3.09  &1.03e-03  &2.71 \\  \hline
1/64   &9.59e-06  &3.48  &2.51e-05  &3.00  &1.45e-04  &2.82 \\  \hline
1/128   &9.52e-07  &3.33  &3.16e-06  &2.99  &1.89e-05  &2.94 \\  \hline
\end{tabular}
\end{table}

\begin{table}[http!]
\center
\caption{Convergence orders for the SBDF-4 scheme (Example~\ref{ex2}).}\label{tab:ex2-k4}
\setlength{\tabcolsep}{3mm}
\begin{tabular}{ |c|c|c|c|c|c|c|}
\hline
$h=\tau$    & $e_0$     &Order  & $e_1$     &Order & $e_\Omega$   &Order\\ \hline
1/16  &1.60e-03  &-  &7.40e-04  &- &1.14e-02 &-\\  \hline
1/32   &1.18e-04  &3.76  &5.80e-05  &3.67  &7.20e-04  &3.98 \\  \hline
1/64   &7.62e-06  &3.95  &3.95e-06  &3.87  &4.24e-05  &4.08 \\  \hline
1/128   &4.78e-07  &3.99  &2.53e-07  &3.97  &2.51e-06  &4.08 \\  \hline
\end{tabular}
\end{table}

\begin{example}[Deformation of a circular disk]\label{ex3}
The exact solution is given by
\begin{equation*}
\Bu= \cos(\pi t/3)
\big(\sin(2\pi x)\sin(2 \pi y),
\cos( 2\pi x)\cos(2 \pi y)\big).
\end{equation*}
The initial domain $\Omega_0$ is same to that in Example~\ref{ex2}, that is, a disk of radius $R=0.15$ and centering at $(0.5,0.5)$. The final time is set by $T=3$.
\end{example}

In this example, the deformation of the domain is even more severe than that in Example~\ref{ex2}.
At $t=T/2$, the middle part of the domain is stretched into a filament (see Fig.~\ref{fig:hT-ex3}).
At this time, the domain attains its largest deformation compared with the initial shape.
In the second half period, the domain returns to its initial shape. Moreover, we adopt a fine mesh with $h=1/256$ to capture the largest deformation of the domain.
Nevertheless, Tables~\ref{tab:ex2-k3} and \ref{tab:ex2-k4} show that quasi-optimal convergence is observed for both the SBDF-3 and SBDF-4 schemes, respectively. This indicates that
\ben
e_0 \sim \tau^k,\qquad
e_1 \sim \tau^k,\qquad
e_\Omega \sim \tau^k,\qquad k=3,4,
\een
hold asymptotically as $\tau=h\to 0$.

Finally, Fig.~\ref{fig:domain-ex2} shows the computational domains formed with the surface-tracking  Algorithm~\ref{alg:mars} during the time evolution procedure. The severe deformations demonstrate the competitive behavior of the finite element method.

\begin{table}[http!]
\center
\caption{Convergence orders for the SBDF-3 scheme (Example~\ref{ex3}).}\label{tab:ex3-k3}
\setlength{\tabcolsep}{3mm}
\begin{tabular}{ |c|c|c|c|c|c|c|}
\hline
$h=\tau$    & $e_0$     &Order  & $e_1$     &Order & $e_\Omega$   &Order\\ \hline
1/32  &9.41e-05  &-  &5.13e-04  &- &7.56e-03 &-\\  \hline
1/64   &3.88e-06  &4.60  &6.13e-05  &3.06  &1.01e-03  &2.90 \\  \hline
1/128   &3.49e-07  &3.47  &7.85e-06  &2.96  &1.28e-04  &2.98 \\  \hline
1/256   &5.25e-08  &2.73  &9.95e-07  &2.98  &1.60e-05  &3.00 \\  \hline
\end{tabular}
\end{table}

\begin{table}[http!]
\center
\caption{Convergence orders for the SBDF-4 scheme (Example~\ref{ex3}).}\label{tab:ex3-k4}
\setlength{\tabcolsep}{3mm}
\begin{tabular}{ |c|c|c|c|c|c|c|}
\hline
$h=\tau$    & $e_0$     &Order  & $e_1$     &Order & $e_\Omega$   &Order\\ \hline
1/32  &9.77e-05  &-  &2.65e-04  &- &2.57e-04 &-\\  \hline
1/64   &9.03e-06  &3.44  &1.27e-05  &4.38  &6.83e-05  &1.91 \\  \hline
1/128   &6.04e-07  &3.90  &7.77e-07  &4.03  &6.39e-06  &3.42 \\  \hline
1/256   &3.84e-08  &3.98  &5.12e-08  &3.92  &4.82e-07  &3.73 \\  \hline
\end{tabular}
\end{table}

\begin{figure}[http!]
	\centering
\includegraphics[width=12cm]{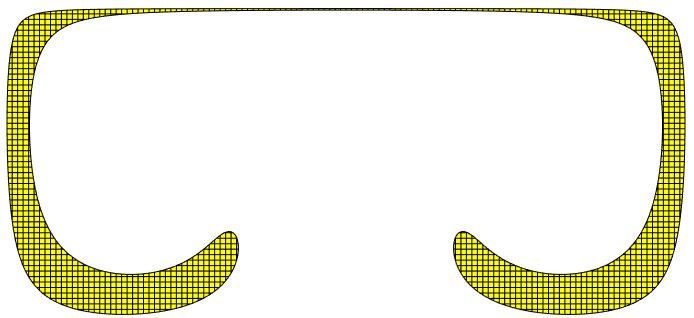}
\caption{The computational domain $\Omega^n_h$ at $t_n=T/2$
$(h=1/128,\;k=4)$.} \label{fig:hT-ex3}
\end{figure}

\begin{figure}[http!]
	\centering
	\subfigure[$t=0$]{
	\includegraphics[width=3cm]{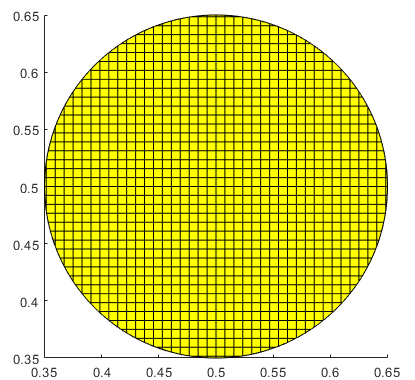}}
	\subfigure[$t=T/8$]{
	\includegraphics[width=5cm]{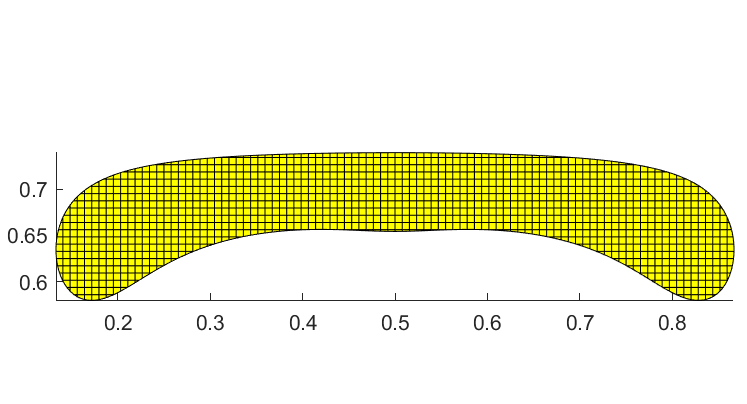}}
	\subfigure[$t=T/4$]{
	\includegraphics[width=5cm]{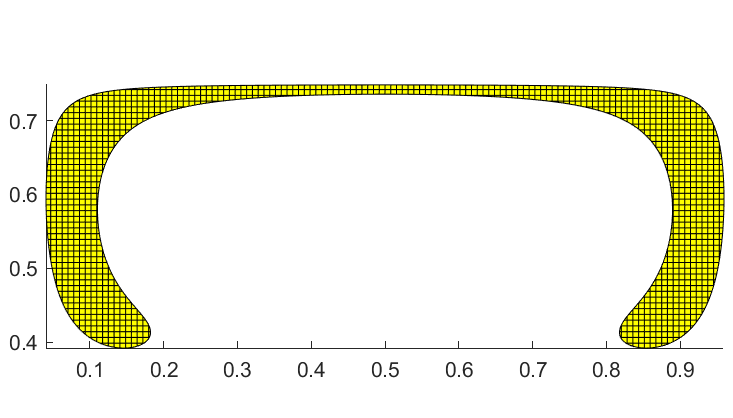}} \\
	\subfigure[$t=T/2$]{
	\includegraphics[width=5cm]{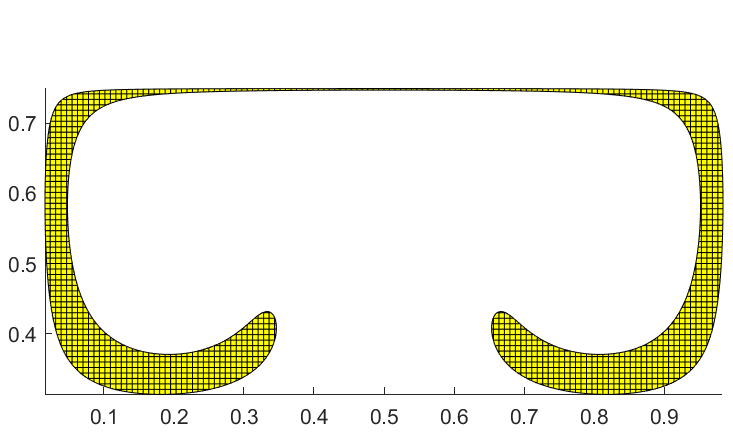}}
	\subfigure[$t=3T/4$]{
	\includegraphics[width=5cm]{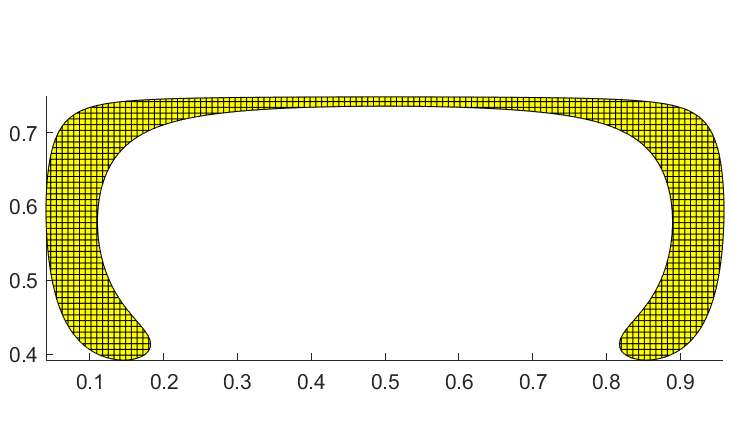}}
    \subfigure[$t=T$]{
	\includegraphics[width=3cm]{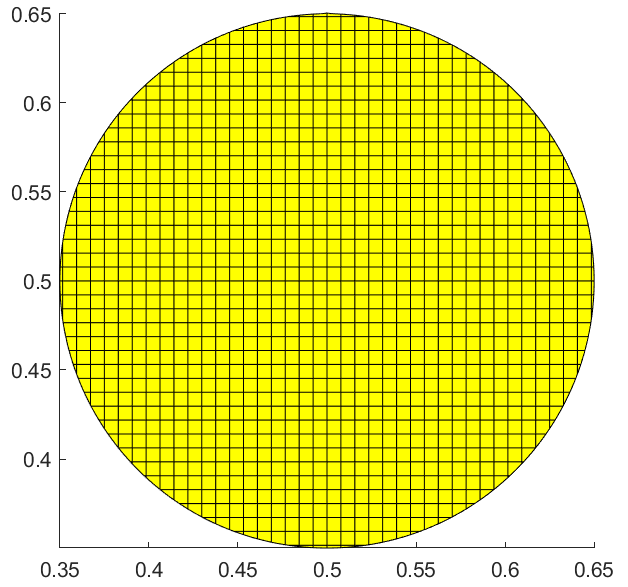}}
\caption{Computational domains $\Omega^n_h$ at different time steps
	$(h=1/128,\;k=4)$.} \label{fig:domain-ex3}
	\end{figure}

\section*{Acknowledgments}

The authors are grateful to professor Qinghai Zhang of Zhejiang University, China. The implementation of Algorithm~\ref{alg:mars} in this paper is based on professor Zhang's interface-tracking code.


\bibliographystyle{amsplain}

\end{document}